\documentclass[11pt]{article} 
\usepackage[utf8]{inputenc} 

\usepackage{geometry} 
\usepackage{graphicx} 
\usepackage{amsmath}
\usepackage{amsthm}
\newcommand{\beq}[1]{\begin{equation}\label{#1}}
\newcommand{\enq}[0]{\end{equation}}
\newtheorem{theorem}{Theorem}

\newtheorem{lemma}[theorem]{Lemma}
\newtheorem{cor}[theorem]{Corollary}

\theoremstyle{definition}
\newtheorem{definition}[theorem]{Definition}

\newtheorem{claim}[theorem]{Claim}

\theoremstyle{remark}

\newtheorem{observe}[theorem]{Observation}

\usepackage{booktabs} 
\usepackage{array} 
\usepackage{paralist} 
\usepackage{verbatim} 
\usepackage{subfig} 
\usepackage[all]{xy}
\usepackage{fancyhdr} 
\usepackage{sectsty}
\allsectionsfont{\sffamily\mdseries\upshape} 

\usepackage[nottoc,notlof,notlot]{tocbibind} 
\usepackage[titles,subfigure]{tocloft} 

\title{Graph decomposition and parity}
\author{Bobby DeMarco \footnotemark $~$ and Amanda Redlich \footnotemark}
\date{}

\begin{document}
\maketitle
\renewcommand{\thefootnote}{\fnsymbol{footnote}}
\footnotetext{ * Supported by the U.S.
Department of Homeland Security under Grant Award Number 2007-ST-104-000006.}
\footnotetext{ $\dag$ Supported by the National Science Foundation under Award No. 1004382.}

\begin{abstract}
Motivated by a recent extension of the zero-one law by Kolaitis and Kopparty, we study the distribution of the number of copies of a fixed disconnected graph in the random graph $G(n,p)$.  We use an idea of graph decompositions to give a sufficient condition for this distribution to tend to uniform modulo $q$.  We determine the asymptotic distribution of all fixed two-component graphs in $G(n,p)$ for all $q$, and we give infinite families of many-component graphs with a uniform asymptotic distribution for all $q$.  
We also prove a negative result, that no recursive proof of the simplest form exists for a uniform asymptotic distribution for arbitrary graphs.
\end{abstract}

\section{Introduction}
A recent paper by Kolaitis and Kopparty \cite{kk} gives an extension of the zero-one law which holds for first-order logic with a parity operator.  The keystone of their proof is that the number of copies (not necessarily induced) of any fixed connected graph is asymptotically uniformly distributed modulo $q$ for any $q$ in the random graph $G(n,p)$.  Other papers have studied this statistic in special cases.  For example, a 2004 paper of Loebl, Matou\v{s}ek, and Pangr\`ac,\cite{LMP} considered the modulo $q$ distribution of triangles, while the 2014 work of Kopparty and Gilmer \cite{justin} gives the distribution of triangles overall.

Here we study the distribution of the number of copies (not necessarily induced) of a fixed \emph{disconnected} graph in $G(n,p)$ modulo $q$.  For convenience, we may say ``the distribution of a graph modulo $q$" to mean the distribution of the number of copies of the graph in $G(n,p)$ with $p$ implicit.  In this paper we will only be speaking of asymptotic distributions, so when discussing distributions we will remove the word asymptotic for brevity.  We say a graph is uniformly distributed if it is uniformly distributed modulo $q$ for all $q$.  We give sufficient conditions for a graph to be uniformly distributed, and we use these conditions to completely characterize the distribution of all 2-component graphs for all $q$.  We then give infinitely large families of uniformly distributed graphs of any component size.

In analyzing these distributions, we developed the concepts of \emph{unique composition} and \emph{decomposition}.  These concepts are related to determining when several connected graphs may be combined to create one large connected graph with certain uniqueness properties.  There are obvious links to the reconstruction conjecture (see \cite{recon1} for a summary), which asks when the subgraphs of a larger graph have slightly different uniqueness properties.  In this paper we give an algorithm for uniquely composing any two feasible graphs, and certain families of three or more graphs.  We also show no generic recursive composition algorithm exists.

The paper is structured as follows.  In the second section we derive a formula for the number of copies of a disconnected graph in a fixed graph $G$ as a function of the number of copies of certain connected graphs in $G$.  In the third section we use this formula to give specific conditions for a disconnected graph to be uniformly distributed.

We show these conditions are satisfied for almost all two-component graphs in the fourth section.  We give an explicit construction for all satisfying graphs.  We also calculate the distribution for all two-component graphs that do not satisfy these conditions.  We then give some examples of infinite families of three or more component graphs that are satisfying.  We conclude this section with a negative result, showing no simple algorithm exists to show a generic graph is satisfying.  In the last section, we discuss areas of further research.
\section{Counting copies}
In this section we give an exact formula for the number of unlabeled copies of a disconnected graph $A=\sqcup_{i=1}^{k} G_{i}$, with non-isomorphic $G_i$, in a host graph $F$.  That is, the number of subgraphs $G_A \subseteq F$ such that $A \simeq G_A$.  For example, if $F$ is $K_4$ and $A$ is two disjoint edges, the number of unlabeled copies of $A$ in $F$ is 3.  Note that this is different from the labeled case, which would give 24 copies.  

Our formula for the number of unlabeled copies of a disconnected graph is given in terms of the number of copies of various connected graphs $H$, and their relationship to the original graph $A$.  Although the formula appears complex, the reasoning behind it is simple.  The main idea is that each copy of $A$ is the product of copies of $G_{1}, G_{2}, \ldots G_{k}$.  Interactions between copies of $G_{i}$ and $G_{j}$ lead to an overcount; the formula uses the principle of inclusion and exclusion to correct for this.  Notice that this is correct only if the $G_i$ are non-isomorphic, hence our assumption that each component is unique.

The simplest case is when $A=G_{1} \sqcup G_{2}$.  For example, let $G_{1}=C_{3}$ and $G_{2}=C_{4}$.  Given fixed host graph $F$, let $N(A)$ be the number of unlabeled copies of $A$ in $F$ (we may also mean the number of unlabeled copies of $A$ in an instance of the random graph $G(n,p)$; this will be clear from context).  Consider $H_i$ as illustrated below.  For ease of discussion, here and throughout the paper, we label the illustrated vertices.  However, the graphs themselves are unlabeled.

We have $$N(A)=N(C_{3})N(C_{4})-N(H_1)-N(H_2)-2N(H_3)-3N(H_4).$$   That is because the total number of disconnected $C_3, C_4$ pairs is the total number of $C_3,C_4$ pairs minus the number of connected pairs.  While $H_1$ and $H_2$ each correspond to exactly one connected pair, $H_3$ is counted twice; once when $C_3=\{1,2,3\}$ and once when $C_3=\{1,3,4\}$.  Similarly, $H_4$ is counted three times: $C_{4}=\{1,2,3,4\}$ or $\{1,2,3,5\}$ or $\{1,5,3,4\}$, and each copy of $C_4$ determines a complementary copy of $C_3$.  (Note that we are counting both induced and non-induced subgraphs; for example, $H_3$ contains no induced $C_4$.)

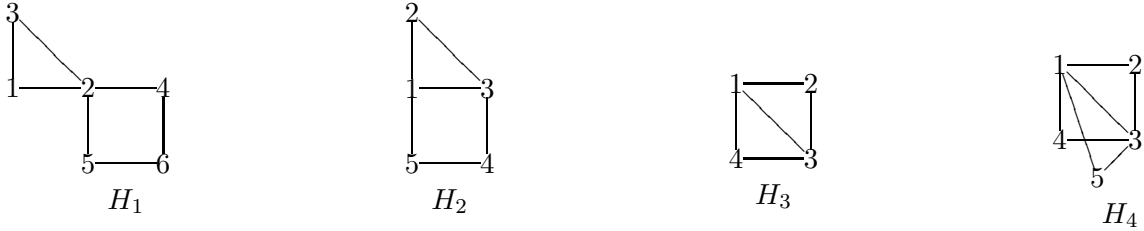
\begin{figure}[htbp]
\begin{minipage}{0.3\linewidth}
\xy
(0,0)*{1}="1"; (10,0)*{2}="2"; (0,10)*{3}="3"; (20,0)*{4}="4"; (10, -10)*{5}="5"; (20,-10)*{6}="6"; (15, -15)*{H_1};
"1";"2" **\dir{-};
"1";"3" **\dir{-};
"2";"3" **\dir{-};
"2";"4" ** \dir{-};
"4";"6" ** \dir{-};
"6";"5" **\dir{-};
"5";"2" **\dir{-};
\endxy
\end{minipage}%
\hfill
\begin{minipage}{.3\linewidth}

\xy
(0,0)*{1}="1"; (10,0)*{3}="3"; (0,10)*{2}="2"; (10, -10)*{4}="4"; (0,-10)*{5}="5"; (5, -15)*{H_2};
"1";"2" **\dir{-};
"1";"3" **\dir{-};
"2";"3" **\dir{-};
"3";"4" ** \dir{-};
"4";"5" ** \dir{-};
"1";"5" **\dir{-};
\endxy
\end{minipage}%
\hfill
\begin{minipage}{.3\linewidth}

\xy
(0,10)*{}; (0,0)*{1}="1"; (10,0)*{2}="2"; (0,-10)*{4}="4"; (10, -10)*{3}="3"; (5, -15)*{H_3};
"1";"2" ** \dir{-};
"2";"3" **\dir{-};
"3";"4" **\dir{-};
"4";"1" **\dir{-};
"1";"3" **\dir{-};
\endxy
\end{minipage}%
\hfill
\begin{minipage}{.3\linewidth}
\xy
(0,0)*{1}="1"; (10,0)*{2}="2"; (0,-10)*{4}="4"; (10, -10)*{3}="3"; (5,-15)*{5}="5"; (0,10)*{}; (8,-20)*{H_4};
"1";"2" ** \dir{-};
"2";"3" **\dir{-};
"3";"4" **\dir{-};
"4";"1" **\dir{-};
"1";"3" **\dir{-};
"1";"5" **\dir{-};
"5";"3" **\dir{-};
\endxy
\end{minipage}
\end{figure}

We now formalize this ``gluing" idea.
\begin{definition}
A tuple $(G_1, G_2, \ldots G_k, H, H_1, H_2, \ldots H_k)$ is a \emph{gluing} of $G_1 \ldots G_k$ if
\begin{itemize}
\item $H$ is a connected graph
\item $H_{1}, \ldots H_{k}$ are subgraphs of $H$
\item $H_{i} \sim G_{i}$ for all $i \in [k]$
\item $\cup_{i=1}^{k} E(H_i)=E(H)$.
\end{itemize}
We occasionally refer to $H$ itself as a gluing.  The tuple $(H_1, \ldots H_k)$ is a \emph{decomposition} of $H$.  If there exists only one tuple $H_1, \ldots H_k$ such that $(G_1, \ldots G_k, H, H_1, \ldots H_k)$ is a gluing, we say that $(G_1, \ldots G_k, H)$ is \emph{uniquely decomposable}.  In this case, we say that $H[S]$ is the unique subgraph of $H$ induced by $\{H_i\}_{i \in S}$.  We occasionally say that $H$ itself is uniquely decomposable, if $G_1, \ldots G_k$ are clear from context.  If there exists an $H$ such that $(G_1, \ldots G_k, H)$ is uniquely decomposable, we say that $\{G_1, \ldots G_k\}$ is \emph{uniquely composable}.
\end{definition}
We often want to count gluings and decompositions.
\begin{definition}
Given $G_1, \ldots G_k$, $s(H)$ is the number of tuples $(H_1, \ldots H_k)$ such that $(G_1, \ldots G_k, H, H_1,\ldots H_k)$ is a gluing.  The set of gluings $\mathbf{H}$ is the family of graphs $H$ such that $s(H) \neq 0$.
\end{definition}
 Using this notation, we re-state a theorem often used in the theory of graph limits, and implied in \cite{kk}:
\begin{theorem}\label{two}
 For a disconnect graph $A=G_1 \sqcup G_2$ and $G_1 \neq G_2$, $$N(A)=N(G_1)N(G_2)-\sum_{H \in \mathbf{H}} s(H)N(H)$$
\end{theorem}
\begin{proof}
The number of copies of $A$ is the number of $G_1, G_2$ pairs overall, less the number of $G_1,G_2$ pairs that intersect.  Each intersecting $G_1,G_2$ pair corresponds to a decomposition of a copy of some $H$, so the total number of intersecting pairs is $\sum_{H \in \mathbf{H}} s(H)N(H)$.
\end{proof}
It is tempting to generalize this to the three-or-more component case as $$N(A) ``=" \prod_{i=1}^{k} N(G_{i})-\sum_{H \in \mathbf{H}} s(H)N(H),$$ but the truth is more complicated.  Along with $H$ that may be decomposed into $G_1, \ldots G_k$, we must consider $H$ that are decomposable into any subset of $G_i$.
For example, if $A$ has components $G_1, G_2, G_3$ then we must be concerned with gluings of the forms $(G_1, G_2, H, H_1, H_2)$, $(G_1, G_3, H, H_1, H_3)$, $(G_2, G_3, H, H_2, H_3)$, and $(G_1, G_2, G_3, H, H_1, H_2, H_3)$.  In order to deal with this complication, we define some new terms.  First, some notation about partitions.
\begin{definition}
Consider the partitions of $[k]$ under partial ordering by refinement, where we use $\pi$ and $\rho$ for partitions of $[k]$ and say $\pi < \rho$ if $\pi$ is a refinement of $\rho$. 
 If $\pi < \rho$, and $T$ is a block of $\rho$, then let $\pi(T)$ be the family of blocks in $\pi$ such that $\cup_{S \in \pi(T)} S=T$.  For example, if $k=4$, $\pi=\{ \{12\} \{3\} \{4\}\}$, and $\rho=\{\{123\} \{4\}\}$, then $\pi(\{123\})=\{ \{12\}\{3\}\}$.
\end{definition}

Now we use partitions to classify gluings.  Each connected component  corresponds to a set in a partition, and the family of gluings is broken into sub-families according to the partitions they generate.
\begin{definition}
Let $\mathbf{H}_{\pi}$ be the family of graphs $H=\sqcup_{S \in \pi} H_{S}$ where each $H_{S}$ may be decomposed (not necessarily uniquely) into $\{H_{i}\}_{i \in S}$.  For example, $\mathbf{H}_{\mathbf{0}}=A$ and any $H \in \mathbf{H}_{\mathbf{1}}$ is a connected graph; $\mathbf{0}$ is the minimum and $\mathbf{1}$ is the maximum element in the partial order. 
\end{definition}
We also count possible decompositions of gluings.  Since we are now considering a broader range of gluings, we must add a subscript to clarify which graphs are being glued.
\begin{definition}
Given $G_1, \ldots G_k$ and $S \subseteq [k]$, let $s_{S}(G')$ be the number of ways the graph $G'$ may be decomposed into copies of $\{G_{i}\}_{i \in S}$
\end{definition}
It may be useful to discuss decompositions into graphs other than our original components $G_{i}$.
\begin{definition}
Given a graph $H_{S}$ for each $ S \in \pi$, $s_{\pi(T)}(G')$ is the number of ways $G'$ may be split into $\{ H_{S} \}_{ S \in \pi(T)}$.  If there exists some $\{H_{S}\}_{S \in \pi}$ such that $s_{\pi([k])}(G') \neq 0$, say that $G'$ is \emph{compatible} with $\pi$.
\end{definition}
Finally, we need to count ``component" decompositions.
\begin{definition}
For any graph $H$, let $p(H)$ be the number of isomorphic permutations of the components of $H$.  That is, if the components of $H$ are $i_1$ copies of some connected graph $B_1$, $i_2$ copies of $B_2$, up to $i_k$ copies of $B_k$, then $p(H)=i_1!i_2!\cdot \ldots\cdot i_k!$.  For example, if $H$ is a five-component graph consisting of three $C_4$ and two $K_5$, then $p(H)=3!2!$: there are 3! ways to decide which $C_4$ is which and 2! ways to decide which $K_5$ is which.
\end{definition}
This notation allows us to give a more general recursion for the number of copies of a graph with an arbitrary number of connected components.  Although the notation is daunting, it is a simple generalization of the ideas in the two-component case.
\begin{theorem}\label{rec}
For a graph $A$=$\sqcup _{i=1}^{k}G_{i}$ with $G_i \neq G_j$ for all $i \neq j$,
\begin{eqnarray*}
N(A)=\prod_{i=1}^{k} N(G_{i})-\sum_{\mathbf{0}<\pi \leq [k]} \sum_{H \in \mathbf{H}_{\pi}}N(\sqcup_{S \in \pi} H_{S})p(H)\prod_{S \in \pi}s_{S}(H_{S})
\\ = \prod_{i=1}^{k} N(G_{i})-\sum_{\mathbf{0}<\pi \leq [k]} \sum_{H \in \mathbf{H}_{\pi}}p(H)\prod_{S \in \pi}s_{S}(H_{S})\left( \prod_{S \in \pi}N(H_{S})-\sum_{\rho > \pi} \sum_{J \in \mathbf{H}_{\rho}}p(J)\prod_{T \in \rho} s_{\pi(T)}(J_{T})p(J_{T})N(\sqcup_{T \in \rho}J_{T})\right)
\end{eqnarray*}
\end{theorem}
\begin{proof}
Now that we have the proper definitions, the proof is short.  As usual, we count the number of copies of $A$ by finding the product of copies of its components, then subtract the overcount.  The ``overcounted" graphs are those in which at least two $G_i$ intersect with each other, i.e. those corresponding to a non-$\mathbf{0}$ partition.  

We also need to be careful about the possibility of intersecting components being isomorphic to a third component.  For instance, returning to our initial $C_3 \sqcup C_4$ example, if we let $A$ be the three-component graph $C_3 \sqcup C_4 \sqcup H_1$, one of the gluings in $H_{\{1,2\} \{3\}}$ is $H_1 \sqcup H_1$.  This gluing should be counted twice because there are two choices for which component is generated by $C_3 \sqcup C_4$.  Therefore we have the first line of the equation.

To see why the second line is true, simply apply the first equation to each $N(\sqcup_{S \in \pi} H_{S})$ term individually.  Now the relevant partitions are those of which $\pi$ is a refinement, and the decompositions are not into $G_i$ but instead $H_S$.
\end{proof}

This theorem gives a recursive algorithm for calculating $N(A)$ for an arbitrary host graph.  With sufficient computing power, then, we could use it to calculate $N(A)$ for the random graph directly.  This theorem can be used as a starting point that will allow us to give explicit counts of a family of graphs, as well as a sufficient condition for graphs to have certain distributions.  The first step is to expand the recursion to get a simpler formula.

\begin{lemma}\label{gen}
For any graph $A = \sqcup_{i=1}^{k} G_{i}$ with $G_i \neq G_j$ for all $i \neq j$, there exist integers $f_A (H)$ for every $H \in \cup_{\pi \leq [k]} \mathbf{H}_{\pi}$ such that $$N(A)=\prod_{i=1}^{k} N(G_{i})-\sum_{\pi \leq [k]} \sum_{H \in \mathbf{H}_{\pi}}\prod_{S \in \pi}N(H_{S})f_{A}(H)$$

\end{lemma}
Note that $f_A$ is uniquely determined; there is no way to write the number of copies of any connected graph in terms of the number of copies of other connected graphs.  
This is clear by inspection, or from \cite{kk}'s proof that the copies of distinct connected graphs are independently distributed.
\section{Distribution of copies}

As mentioned in the introduction, \cite{kk} proves that, for any constants $p$ and $q$, any $i < q$, and any connected graph $G_0$, the probability of $G(n,p)$ having $i$ copies of $G_0$ modulo $q$ tends to $1/q$ as $n$ tends to infinity.  That is, the distribution of a connected graph in the random graph tends to uniform modulo $q$.  We give exact distributions for the number of copies of any disconnected $G_0$ in $G(n, p)$ modulo $q$ in this section by combining the formulas of the previous section with these results on connected graphs.

The previous section gives exact expressions for the number of copies of a disconnected graph in a particular graph.  The formulas are often difficult to implement.  However, since our goal is the distribution of the count, rather than its exact value, the preceding formulas are enough.  To study the distributions of disconnected graphs, we first recall Theorem 3.2 in \cite{kk}, which we restate here:
\begin{theorem}\label{con}
For any $q >1$ and $p \in (0,1)$, and any family of distinct finite connected graphs $F_1, \ldots F_l$, the distribution of $(N(F_1), \ldots N(F_l))$ modulo $q$ is $2^{-\Omega(n)}$ close to uniform over $[q]^{l}$.
\end{theorem}


In other words, there are no relations between the number of copies of \emph{distinct} connected subgraphs.  Therefore we may use \ref{rec} without being concerned about possible dependencies in the distribution.  Combining Theorem \ref{con} and Lemma \ref{gen} produces the following corollary (note that for any connected graph, e.g. $H \in \mathbf{H}_{\mathbf{1}}$, $p(H)=1$).

\begin{cor}\label{ets}
 Given  $A=\sqcup_{i=1}^{k} G_i$ with distinct $G_i$, if there exists some $H \in \mathbf{H}_{\mathbf{1}}$ such that $f_A(H)$ is relatively prime to $q$ and $H$ is not a gluing of $\{G_i\}_{i \in S}$ for any $S \subset [k]$, then the distribution of $N(A)$ in $G(n,p)$ is $2^{-\Omega(n)}$-close to uniform modulo $q$.
\end{cor}

The rest of this section is concerned with finding, for a given $A$, an $H$ such that $f_A(H)=\pm 1$.  This is enough to show that, for any $q$, $N(A)$ is distributed uniformly modulo $q$.
\begin{definition}
Given $A=\sqcup_{i=1}^{k} G_{i}$ and a uniquely decomposable gluing \\$H$ with decomposition $(H_{1}, \ldots H_{k})$, the \emph{structure graph} of $H$, denoted by $T(H)$, is the graph whose vertices are $[k]$ and edges are pairs ${i,j}$ such that $H_{i} \cap H_{j}$ is non-empty.
\end{definition}
\begin{definition}
Given $A=\sqcup_{i=1}^{k} G_{i}$, a gluing $H$ is \emph{tree-like} if it is uniquely decomposable and its structure graph, $T(H)$, is a tree.
\end{definition}

\noindent We now show
\begin{theorem}\label{tree}
For any graph $A=\sqcup_{i=1}^{k}G_{i}$ with distinct components and tree-like gluing $H$,
$f_A(H)=(-1)^{k-1}$.
\end{theorem}
\begin{proof}
The proof is by strong induction.  When $k=2$ the statement follows from Theorem \ref{two}: when $k=2$ any uniquely decomposable gluing is tree-like.

Now consider $k\geq 3$.  By the first line of Theorem \ref{rec},
$$N(A)=\prod_{i=1}^{k} N(G_{i})-\sum_{0<\pi \leq [k]} \sum_{J \in \mathbf{H}_{\pi}}N(\sqcup_{S \in \pi} J_{S})\prod_{S \in \pi}s_{S}(J_{S}).$$
The induction hypothesis applied to $\sqcup_{S \in \pi} J_{S}$ implies that
\beq{fAH1}
f_A(H)=-\sum_{0<\pi \leq [k]} \sum_{J \in \mathbf{H}_{\pi}}f_J(H)\prod_{S \in \pi}s_{S}(J_{S})
\enq

Note that for any $J$ with $f_J(H)\neq 0$, the unique decomposability of $H$ gives that there is exactly one $0<\pi\leq [k]$ such that $J\in \mathbf{H_\pi}$.  Let us call this partition $\pi(J,H)$.  Furthermore, it also gives that $p(J)=1$ for any such $J$.  We will show

\begin{claim}\label{cl1}
For any $r-$component $J$, if $f_J(H)\neq 0$, then $J_S=H[S]$ for all $S\in \pi(J,H)$ and $f_J(H)=-1^{r-1}$.
\end{claim}

\begin{claim}\label{cl2}
If $f_J(H)\neq 0$ then $\prod_{S\in \pi(J,H)} s_S(J_S)=1$.
\end{claim}

\begin{claim}\label{cl3}
For $0<\pi\leq [k]$ the number of $J\in \mathbf{H}_\pi$ such that $J_S=H[S]$ $\forall S\in \pi$ is one if $H$ is compatible with $\pi$ and zero otherwise.  If $H$ is compatible with $\pi$, call $\pi$ {\em $H-$good}.
\end{claim}

\begin{claim}\label{cl4}
There are $\binom{k-1}{r-1}$ $H-$ good partitions $\pi$ consisting of $r$ sets.
\end{claim}

Combining Claims \ref{cl1}-\ref{cl4} with \eqref{fAH1} we immediately have (think of $r$ as the number of components of $J$, or equivalently number of sets in the partition $\pi$ related to $J$)
\begin{align*}
f_A(H)&=-\sum_{r=1}^{k-1} \binom{k-1}{r-1}(-1)^{r-1}\\
&=-1^{k-1}.
\end{align*}

\noindent {\em Proof of Claim} \ref{cl1}
The first part of the claim holds because $H$ is uniquely decomposable.  Note that $H$ is a tree-like gluing of the components of $J$.  Thus the second part of the claim is an application of our inductive hypothesis.

\medskip
\noindent {\em Proof of Claim} \ref{cl2}
This follows from the unique decomposability of $H$.

\medskip
\noindent {\em Proof of Claim} \ref{cl3}
$J_S$ must be a connected graph.  So if $H[S]$ is connected for all $S\in \pi$ then $J=\cup_{S\in \pi} H[S]$ is clearly the only $J\in \mathbf{H}_\pi$ such that $J_S=H[S]$ $\forall S\in \pi$.  If $H[S]$ is disconnected, then $H$ was not compatible with $\pi$.  Thus there are zero such graphs.

\medskip
\noindent {\em Proof of Claim} \ref{cl4}
Consider the natural mapping from a partition $0<\pi\leq [k]$ to the set $E(T(H)\setminus F)$ where $F=\cup_{S\in \pi} H[S]$.  This mapping defines a bijection from the $H-$good partitions $\pi$ consisting of $r$ distinct sets and
the set of subgraphs of $T(H)$ with $r-1$ edges.

\end{proof}
\section{Specific examples}
Here we give some applications of the theorems of the previous section.    We begin with a complete characterization of the distribution of all two-component graphs, together with explicit constructions.  We then give several families of graphs that have tree-like gluings, and therefore by Theorem \ref{tree} are uniformly distributed.  Finally, we show that no recursive proof of the simplest form exists for a uniform asymptotic distribution for arbitrary graphs.

\subsection{Two component graphs}
Any uniquely decomposable gluing of two graphs must be tree-like.  So one way to show that $N(A)$ is uniformly distributed for some two component graph $A$ would be to give a construction of a uniquely-decomposable $H$.  
In fact, such a construction exists for all two-component graphs \emph{except} a few trivial cases.
\begin{theorem}\label{alg}
If $G_1 \neq G_2$, neither $G_1$ nor $G_2$ is a single vertex, $\{G_1, G_2\} \neq \{P_{1}, P_{2}\}$ and $\{G_1, G_2\} \neq \{P_{1}, P_{3}\}$, there exists a graph $H$ such that
\\$(G_1, G_2, H, H_1, H_2)$ is a uniquely decomposable gluing and $H \neq G_1, G_2$.  Furthermore, $H$ may be constructed explicitly.
\end{theorem}
In order to describe the construction of $H$, we define a few new terms.  $H$ will be created by taking two graphs and ``gluing" them together.
\begin{definition}\label{glue}
Given $G_{1}$ and $G_{2}$ and vertices $v_{1} \in V(G_{1})$ and $v_{2} \in V(G_{2})$, \emph{to glue} $v_{1}$ and $v_{2}$ to create a new graph $H_v$ is the natural identification: $$V(H_v)=(V(G_{1})\setminus\{v_1\})\cup (V(G_{2})\setminus \{v_2\}) \cup \{v\}$$ and $$E(H_v)=\{ \{x,y\} | \{x,y\} \in E(G_{1}) \cup E(G_{2})\} \cup\{ \{x,v\} | \{x,v_{1}\}\in E(G_{1})\}\cup \{ \{v,y\} |\{v_{2},y\} \in E(G_{2})\}.$$  
Given $G_1$ and $G_2$ and edges $\{u_1, v_1\}=e_1 \in G_1$ and $\{u_2,v_2\}=e_2 \in G_2$, \emph{to glue} $e_1$ and $e_2$ to create a new graph $H_e$ is the natural identification: $$V(H_e)=(V(G_{1})\setminus\{u_1,v_1\})\cup (V(G_{2})\setminus \{u_2, v_2\}) \cup \{u, v\}$$ and
\begin{eqnarray*}
E(H_v)=\{ \{x,y\} | \{x,y\} \in E(G_{1}) \cup E(G_{2})\}\cup\{ \{x,u\} | \{x,u_{1}\}\in E(G_{1})\cup\{ \{x,u\} | \{x,u_{2}\}\in E(G_{2})
\\
\cup\{ \{x,v\} | \{x,v_{1}\}\in E(G_{1})\}\cup \{ \{x,v_2,\} |\{x, v_2\} \in E(G_{2})\}\cup\{\{u,v\}\}.
\end{eqnarray*}
Notice that there are two possible gluings along an edge, as there are two possible pairings of the endpoints of the edges.  Constructions in this paper work for an arbitrary pairing of endpoints.
\end{definition}

Another difficulty is deciding where to glue two graphs.  To describe gluing locations, we consider the underlying connectivity structure of each graph.  We say the \emph{block degree} $b_{G}(v)$ of a vertex $v\in G$  is the number of components generated by the removal of $v$, i.e. $b(v)=$ (number of components of $G-v$) - (number of components of $G$).  Note that $b_{G}(v) >0$ if and only if $v$ is a cut vertex.  So every connected graph has at least two vertices of block degree $0$, which we will call \emph{block-leaves}.  Let $B(G)=\max_{v \in V(G)} b_{G}(v)$.

Throughout the following discussion we let $H$ refer to the graph created by gluing together $G_1$ and $G_2$ at either $v_1$ and $v_2$, or $e_1$ and $e_2$, as discussed in Definition \ref{glue}.   $H_1,H_2$ will be an arbitrary decomposition of $H$. That is, $H_i$ may be the original graph $G_i$, or it may be a different image of $G_i$ in $H$.   Note that, if $H$ is formed by gluing at a vertex, then $H_1 \cap H_2$ is a single vertex.  Similarly, if H is formed by gluing at an edge, $H_1 \cap H_2$ is a single edge.  We begin with a few observations about the block degree.

\begin{observe}
If $H$ is made by gluing together $G_1$ and $G_2$ at a vertex, then in any decomposition $H_1,H_2$, with $z=H_1\cap H_2$, for all $x\neq z \in H_1$,
$$b_{H_1}(x)= b_{H}(x)$$
and for all $x\neq z \in  H_2$,
$$b_{H_2}(x)= b_{H}(x).$$  Furthermore, $b_{H_1}(z),b_{H_2}(z)\leq b_H(z)$.
\end{observe}

\begin{observe}
If $H$ is formed by gluing together $G_1$ and $G_2$ at an edge, then in any decomposition $H_1,H_2$, for all $x \in H_1\cap \bar{H_2}$,
$$b_{H_1}(x)= b_{H}(x)$$
and for all $x\in H_2\cap \bar{H_1}$,
$$b_{H_2}(x)= b_{H}(x).$$  For $x\in H_1\cap H_2$, $b_{H_1}(x),b_{H_2}(x)\leq b_H(v)$.
\end{observe}

With these definitions and observations in hand, we begin the proof of Theorem \ref{alg}.
\begin{proof}
Without loss of generality, we may assume $B(G_{2})\geq B(G_{1})$ and, if $B(G_{2})=B(G_1)$, then $|V(G_{1})|\leq |V(G_{2})|$.  Let $S_{i}$ be the set of vertices in $G_{i}$ of block-degree $B(G_{i})$.  We split graph pairs into six cases, according to their block degrees and other traits, and give a construction for each case. \\
\\
Case A:  $B(G_{2})> B(G_{1})$ and $B(G_2)>1$.  In this case, glue a block-leaf at maximum distance from $S_2$ to any block-leaf in $G_1$ to create $H$.

Suppose $H_2 \cap G_1 \neq \{v\}$.  $H_2\cap G_1$ is a connected graph:  if not, then because $H_2$ is connected there is a path between any two disconnected components of $H_2 \cap G_1$ within  $G_2$.  But any such path must begin and end at $v$, and therefore $H_2 \cap G_1$ itself was connected.

Thus there are at least two block-leaf vertices in $H_2 \cap G_1$, hence at least one block-leaf in $H_2\cap G_1$ not equal to $v$.  Choose one such vertex and label it $w$.  Let $R$ be the set of vertices in $H_2$ such that $\phi(R)= S_2$.  Since $b_H(v)=2<B(G_2)$, by Observation 1, we must have $R=S_2$.  Then $d(w,R)=d(w,S_2)>d(v,S_2)$  a contradiction.  Thus $H_2 \cap G_1 =\{v\}$, and the decomposition is unique.
\\
\\
Case B: $B(G_2)=B(G_1)$.  In this case, glue a block-leaf in $G_2$ to any vertex in $S_1$.  Since $b_H(v)>B(G_2)$, Observation 1 implies that $H_1\cap H_2=\{v\}$.  Therefore each component of $H\setminus\{v\}$ must be entirely contained within $H_1$ or $H_2$.  Now we use that $|V(G_2)|\geq |V(G_1)|$ to conclude $H$ is uniquely decomposable.
\\
\\


Case C: $B(G_2)=1$, $B(G_1)=0$, $G_1 \neq K_2$, and there exists a block within $G_2$ that is not isomorphic to $G_1$.  Because $B(G_2)=1$, each cut-vertex in $G_2$ connects two blocks.  Therefore we can look at $G_2$'s structure as a tree $T_{G_2}$, with vertices $v\in T_{G_2}$ corresponding to each block in $G_2$, and edges $e\in T_{G_2}$ corresponding to each cut-vertex $u \in G_2$.  

Color a vertex in this tree black if its block in $G_2$ is isomorphic to $G_1$ and white otherwise.  Let a special path in $T_{G_2}$ be any path $v_1, v_2, \ldots v_k$ in $T_{G_2}$ such that $v_1$ is white and $v_2, \ldots v_k$ are black.  Notice that by definition there must be a white vertex, therefore there must be at least one special path (possibly consisting of just one vertex).  

Let $u_1, u_2, \ldots u_M$ be a special path of maximal length.  Create $H$ by gluing a vertex of $G_1$ to a vertex of $u_M$ not in $u_{M-1}$. 
 $H$ now has a longer special path than $G_2$.  This is uniquely decomposable because any decomposition of $H$ must break this new longest special path by splitting a single vertex.  The only vertex that will split $H$ into a copy of $G_1$ and a copy of $G_2$ is the glued vertex; any other will generate two graphs, each of which contains at least two blocks.
\\
\\
Case D: $B(G_2)=1$, $B(G_1)=0$, $G_1 \neq K_2$, and $G_2$ consists of blocks isomorphic to $G_1$.  

Glue any edge of $e_1$ of $G_1$ to any edge in $e_2$ of $G_2$ whose end-vertices both have block degree 0.  Call this new glued block $G_g\subset H$, and the glued edge $e$.  We show that $H$ is uniquely decomposable by focusing on $G_g$ and how it must interact with any decomposition $H_1, H_2$.

First notice that $G_g$ cannot be entirely contained within $H_1$: $G_g$ has too many vertices.  Furthermore, $G_g$ cannot be entirely contained within $H_2$: $G_g$ is not isomorphic to $G_1$, so it cannot be a block in $H_2$.   Therefore $G_g$ contains the intersection of $H_1$ and $H_2$.  We also know that $H_1$ is entirely within $G_g$:  If $H_1$ has vertices both in and out of $G_g$, it would contain a cut vertex.  Also notice a simple counting shows that the intersection of $H_1$ and $H_2$ must consist of a single edge.  

Therefore, we have established that $H_1 \subset G_g$, $G_g \cap H_2 \neq \emptyset$, and $H_1\cap H_2 =e'$ for some edge $e' \in G_g$.  We now show that $e'=e$.  

Suppose not.  Notice that, by construction, there are no edges other than $e$ in $H$ that have one end vertex in $G_1$ and the other end vertex in $G_2$.  Therefore, if $e'\neq e$,  $e'$ must be entirely within $G_1$ or entirely within $G_2$.    Thus splitting along $e'$ means splitting $G_1$ or $G_2$.  If $G_1$ is split, one of $H_1, H_2$ is a proper subgraph of $G_1$, which is a contradiction.  If $G_2$ is split, then $G_1$ is a proper subgraph of both $H_1$ and $H_2$, which is a contradiction.
\\
\\
Case E: $B(G_2)=1$, $D(G_2)>2$ and $G_1=K_2$.  In this case, if $G_2$ contains a vertex of degree one, glue a vertex of $G_1$ to a leaf at maximum distance from $S_2'$, the set of vertices of $G_2$ of maximum degree.  Let $w$ be the vertex in $G_1$ not glued to $G_2$.  Note that $d(w,S_2')=d(v,S_2')+1$ which is strictly greater than the distance from $x$ to $S_2'$ for any leaf $x\in G_2$.  Thus $w\not\in H_2$ and we conclude the decomposition of $H$ is unique.

If $G_2$ does not contain any vertices of degree one, then glue any vertex of $G_1$ to any vertex of $G_2$.  Then $w$, the vertex in $G_1$ not glued to $G_2$, must be in $H_1$ and we conclude the decomposition of $H$ is unique.
\\
\\
Case F: $G_2=P_k$, $k>3$, and $G_1=K_2$.  In this case, glue a vertex of $K_2$ to the third vertex along the path $P_k$.  It is clear that this graph is uniquely decomposable.

\end{proof}
In fact, this construction covers almost all uniformly-distributed two-component graphs.  We fully characterize the distributions of two-component graphs by combining Theorem \ref{alg} with some examination of a few special cases.  
\begin{theorem}
For every graph $A$ with connected components $G_{1}\neq G_{2}$
\begin{itemize}
\item If neither $G_1$ nor $G_2$ is a single vertex, and $\{G_{1},G_{2}\}\neq \{P_{1}, P_{2}\}, \{P_{1}, P_{3}\}$ (where $P_{i}$ is the path with $i$ edges), $N(A)$ is $2^{-\Omega(n)}$-close to uniformly distributed in $G(n, p)$ modulo any $q$.
\item If $A=P_{1} \sqcup P_{2}$, $N(A)$ is $2^{-\Omega(n)}$-close to uniformly distributed modulo $q$
\item If $A=P_{1}\sqcup P_{3}$, $N(A)$ is $2^{-\Omega(n)}$-close to uniformly distributed modulo $q$ if and only if $q$ is odd.  If $q$ is even, $N(A)$ is $2^{-\Omega (n)}$-close to being $$P(N(A)\equiv 2i)=3/2q$$ and $$P(N(A)\equiv 2i+1)=1/2q$$ for all $i \in \{0, \ldots q/2\}$.
\item If, without loss of generality, $G_1=K_1$, $N(A)$ is $2^{-\Omega(n)}$-close to being $$P(N(A) \equiv il)=l/q,$$ where $l=\text{gcd}(q, n-|V(G_2)|)$

\end{itemize}
\end{theorem}
\begin{proof}
The first item follows directly from Theorem \ref{alg}, Theorem \ref{tree}, and Corollary \ref{ets}.  The next cases, $P_1 \sqcup P_2$, $P_1 \sqcup P_3$, and $K_1 \sqcup G_2$, are solved by direct computation.

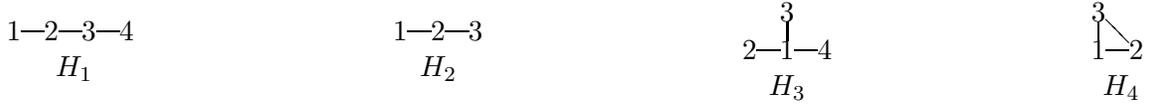
\begin{figure}[htbp]
\begin{minipage}{0.3\linewidth}
\xy
(0,0)*{1}="1"; (5,0)*{2}="2"; (10,0)*{3}="3"; (15,0)*{4}="4"; (8, -5)*{H_1};
"1";"2" **\dir{-};
"2";"3" **\dir{-};
"3";"4" ** \dir{-};
\endxy
\end{minipage}
\hfill
\begin{minipage}{.3\linewidth}
\xy
(0,0)*{1}="1"; (5,0)*{2}="2"; (10,0)*{3}="3"; (5, -5)*{H_2};
"1";"2" **\dir{-};
"2";"3" **\dir{-};
\endxy
\end{minipage}
\hfill
\begin{minipage}{.3\linewidth}
\xy
(0,0)*{2}="2"; (5,0)*{1}="1"; (5,5)*{3}="3"; (10,0)*{4}="4"; (5, -5)*{H_3};
"1";"2" **\dir{-};
"1";"3" **\dir{-};
"1";"4" **\dir{-};
\endxy
\end{minipage}
\hfill
\begin{minipage}{.3\linewidth}
\xy
(0,0)*{1}="1"; (5,0)*{2}="2"; (0,5)*{3}="3"; (3,-5)*{H_4};
"1";"2" **\dir{-};
"2";"3" **\dir{-};
"3";"1" **\dir{-};
\endxy
\end{minipage}
\end{figure}

First consider the case $A_1=P_1 \sqcup P_2$.  All gluings $H \in \mathbf{H}$ of $P_1$ and $P_2$ are illustrated above.  Note that $H_2=P_2$, $s(H_1)=2$, $s(H_2)=2$, $s(H_3)=3$, and $s(H_4)=3$.  Therefore we have $$N(A_1)=(N(P_1)-2)N(P_2)-2N(H_1)-3N(H_3)-3N(H_4).$$  By Theorem \ref{con}, we know that the tuple $(N(P_1), N(P_2), N(H_1), N(H_3), N(H_4))$ is $2^{-\Omega(n)}$-close to being uniformly distributed over $\mathbf{Z}_{q}^{5}$.  Therefore $N(A_1)$ itself is $2^{-\Omega(n)}$-close to being uniformly distributed.

Now consider the case $A_2=P_1 \sqcup P_3$.  All gluings $H \in \mathbf{H}$ of $P_1$ and $P_3$ are illustrated below.  Note that $H_2=P_3$, $s(H_1)=2$, $s(H_2)=3$, $s(H_3)=2$, $s(H_4)=4$, and $s(H_5)=2$.  Therefore we have $$N(A_2)=(N(P_1)-3)(N(P_3))-2N(H_1)-2N(H_3)-4N(H_4)-2N(H_5).$$  Again, Theorem \ref{con} and some basic modular arithmetic are enough to generate the distributions modulo $q$ in each case.
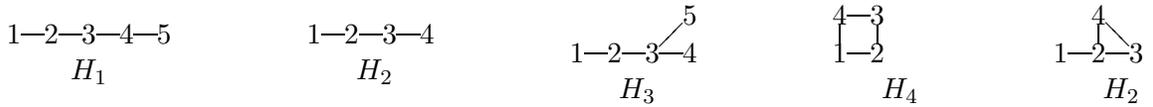
\begin{figure}[htbp]
\begin{minipage}{0.3\linewidth}
\xy
(0,0)*{1}="1"; (5,0)*{2}="2"; (10,0)*{3}="3"; (15,0)*{4}="4"; (20, 0)*{5}="5"; (10, -5)*{H_1};
"1";"2" **\dir{-};
"2";"3" **\dir{-};
"3";"4" ** \dir{-};
"4";"5" ** \dir{-};
\endxy
\end{minipage}
\hfill
\begin{minipage}{.3\linewidth}
\xy
(0,0)*{1}="1"; (5,0)*{2}="2"; (10,0)*{3}="3"; (15,0)*{4}="4"; (8, -5)*{H_2};
"1";"2" **\dir{-};
"2";"3" **\dir{-};
"3";"4" ** \dir{-};
\endxy
\end{minipage}
\hfill
\begin{minipage}{.3\linewidth}
\xy
(0,0)*{1}="1"; (5,0)*{2}="2"; (10,0)*{3}="3"; (15,0)*{4}="4"; (15,5)*{5}="5"; (8, -5)*{H_3};
"1";"2" **\dir{-};
"2";"3" **\dir{-};
"3";"4" ** \dir{-};
"3";"5" ** \dir{-};
\endxy
\end{minipage}
\hfill
\begin{minipage}{.3\linewidth}
\xy
(0,0)*{1}="1"; (5,0)*{2}="2"; (0,5)*{4}="4"; (5, 5)*{3}="3";(0,5)*{}; (8,-5)*{H_4};
"1";"2" ** \dir{-};
"2";"3" **\dir{-};
"3";"4" **\dir{-};
"4";"1" **\dir{-};
\endxy
\end{minipage}
\hfill
\begin{minipage}{.3\linewidth}
\xy
(0,0)*{1}="1"; (5,0)*{2}="2"; (10,0)*{3}="3"; (5,5)*{4}="4"; (8, -5)*{H_2};
"1";"2" **\dir{-};
"2";"3" **\dir{-};
"2";"4" ** \dir{-};
"4";"3" ** \dir{-};
\endxy
\end{minipage}
\end{figure}

Now, consider the case $A_3=K_1 \sqcup G_2$.  It is clear that $N(A)=(n-|V(G_2)|)N(G_2)$. Once more, Theorem \ref{con} and some basic modular arithmetic are enough to generate the distribution modulo $q$.

\end{proof}

\subsection{Tree-like gluings}
Graphs with more than two components are harder to work with using the methods of the previous section.  As the number of components increases, the possible gluings and decompositions also increase. Nevertheless, there are some families of multi-component graphs that admit a recursive construction.
\begin{theorem}\label{diff}
If $A=\sqcup_{i=1}^{k} G_i$ and there do not exist $i\neq j$ such that $G_i$ is a subgraph of $G_j$, then there exists $H_A$ a tree-like gluing of $\{G_i\}$ such that $H_A \neq G_i$, and $N(A)$ is $2^{-\Omega(n)}$-close to being uniformly distributed modulo $q$ for all $q$.
\end{theorem}
\begin{proof}
Without loss of generality, let the graphs be listed in non-decreasing order by diameter.  Let $u_i$ and $v_i$ be vertices of $G_i$ at maximal distance from each other.  Then let $H_A$ be the graph constructed by gluing $v_i$ to $u_{i+1}$.  The structure graph is clearly a tree.

It is also uniquely decomposable, by induction:  Suppose this construction is uniquely decomposable for all $k<n$.  Now consider $H_A$ for $k=n$.  
Suppose there exists some decomposition so that $G_n \neq H_n$.  Consider $v_n$.  If $v_n \in H_i$ for $i \neq n$, then because $D(G_n)\geq D(G_i)$, all vertices in $H_i$ must be within $G_n$.  That contradicts our initial condition that no graphs is a subgraph of another, so it cannot happen.

Therefore $x \in H_n$.  Again, by a diameter argument, $G_n=H_n$.  Therefore any decomposition of $H_A$ must fix $G_n$.  The graph $H_A - H_n$ is the construction for $G_1, \ldots, G_{n-1}$, so by induction it is also uniquely decomposable.

Because $H_A$ has a tree structure and is uniquely decomposable, by Theorem \ref{tree} $N(A)$ is $2^{-\Omega(n)}$-close to uniformly distributed modulo $q$.
\end{proof}
The reader can generate many corollaries of Theorem \ref{diff}, using any subgraph-free family of graphs.  

The same``path-like" gluing shows another family of multi-component graphs is also uniquely decomposable.
\begin{theorem}
If $A=\sqcup_{i=1}^{k} G_i$ and the $G_i$ are distinct and two-connected, then $N(A)$ is $2^{-\Omega(n)}$-close to being uniformly distributed modulo $q$ for all $q$.
\end{theorem}
\begin{proof}
Similarly to the proof of Theorem \ref{diff}, glue the graphs together at the vertices of maximum distance from each other.  $H_A$ contains $k-1$ vertices of block-degree 1, which are exactly the glued vertices.  Notice that any decomposition must split all block-degree 1 vertices into two block-degree 1 vertices; therefore there is exactly one decomposition and the graph is uniquely decomposable.
\end{proof}

\subsection{No generic gluing exists}

The previous constructions used a recursive process to create a uniquely decomposable $H$ for $G$ satisfying certain conditions.  A natural goal would be to find a generic recursive process to generate  $H$ for arbitrary $G$.  However, no such construction exists.
\begin{theorem}\label{no}
There does not exist a generic recursive construction algorithm $C$ that, for all $k$ and distinct $G_1, \ldots, G_k$, generates a uniquely decomposable $H_k$.  That is, there does not exist an algorithm $C$ that, given $G_1, \ldots G_k$ in that order, constructs uniquely decomposable $H_k$ by first calling $C$ on $G_1, \ldots, G_{k-1}$ to generate $H_{k-1}$, and then calling $C$ on $H_{k-1}, G_{k}$.
\end{theorem}
\begin{proof}
Suppose there did exist such a recursive $C$.  Let $C(G_1, \ldots G_{k-1})=H_{k-1}$.  If $G_1,\ldots, G_{k-1}$ can be glued together as a proper subgraph of $H_{k-1}$, then $C$ cannot construct a uniquely decomposable $H_{k}$ on input $G_1, \ldots, G_{k-1}, H_{k-1}$.  We note that, for example, $G_1\subseteq G_2$ is enough to give that $G_1,\ldots,G_{k-1}$ can be glued together as a proper subgraph of $H_{k-1}$.

We also point out that ordering is important to this proof; as far as we know, it is possible that an algorithm exists that, given $G_1, \ldots G_k$, first analyzes the individual graphs, then calls them in a particular order $G_{j_1}, \ldots G_{j_k}$.
\end{proof}

\section{Open questions}
There are two main open questions. What disconnected graphs are distributed uniformly?  What families of connected graphs are uniquely or tree-like composable?

Theorem \ref{ets} gives us one means of studying graph distributions.  However, it is not the case that graphs are uniformly distributed exactly when they have tree-like compositions.  (Recall that $P_1 \sqcup P_2$ is uniform but is not uniquely composable.)  It is possible that a more sophisticated analysis of the formula in Theorem \ref{rec} could give a different sufficient condition for graphs to be uniformly distributed.

We have fully characterized the two-component graphs that are uniquely composable, and hence admit tree-like compositions.  We believe an approach similar to the two-component construction given here also works for the three-component case.  However, increasing the number of components significantly complicates the analysis, and the number of cases is over twenty.  We are currently developing a simpler construction for three components.

Of course, the ultimate goal is to completely characterize the uniquely composable and tree-like composable graphs with any number of components.  We suspect that many, if not all,  graphs admit such compositions.  Theorem \ref{no} indicates a recursive approach does not work in general, but a different type of algorithm may succeed.  Even a non-constructive proof of the existence of uniquely decomposable or tree-like graphs would be interesting.


\end{document}